\documentclass[10pt]{amsart}
\usepackage{verbatim}
\usepackage{eucal,url,amssymb,stmaryrd,enumerate,amscd,}
\usepackage[pagebackref,colorlinks=true]{hyperref}
\usepackage{amsfonts}
\usepackage{amsmath,amsthm,amssymb,amscd,enumerate,eucal,url,stmaryrd}

\usepackage[margin=1in]{geometry}
\numberwithin{equation}{section}
\newtheorem{thrm}{Theorem}[section]
\newtheorem{lemma}[thrm]{Lemma}

\newtheorem{cor}[thrm]{Corollary}
\newtheorem{dfn}[thrm]{Definition}

\newtheorem{rmrk}[thrm]{Remark}

\newtheorem{conv}[thrm]{Convention}

\newcommand{\QH}{\boldsymbol {G\,(\mathbb{H})}}

\newcommand{\dxa}[1]{\frac {\partial {#1}} {\partial x_{\alpha}} }
\newcommand{\dta}[1]{\frac{\partial {#1}}{\partial t_{\alpha}}}
\newcommand{\dya}[1]{\frac{\partial {#1}}{\partial y_{\alpha}}}
\newcommand{\dza}[1]{\frac{\partial {#1}}{\partial z_{\alpha}}}

\newcommand{\dx}[1]{\frac{\partial {#1}}{\partial x}}
\newcommand{\dy}[1]{\frac{\partial {#1}}{\partial y}}
\newcommand{\dz}[1]{\frac{\partial {#1}}{\partial z}}

\newcommand{\LieQ}{\mathcal{L}_Q\, }

\begin{document}

\begin{abstract}
{\ A complete solution to the quaternionic contact Yamabe equation
on the qc sphere of dimension $4n+3$ as well as on the
quaternionic Heisenberg group is given. A uniqueness theorem for
the qc Yamabe problem in a compact locally 3-Sasakian manifold is
shown.}
\end{abstract}

\title[Solution of the qc Yamabe equation on the 3-Sasakian sphere]
{Solution of the qc Yamabe equation on a 3-Sasakian manifold and
the quaternionic Heisenberg group}
\date{\today}
\author{Stefan Ivanov}
\address[Stefan Ivanov]{University of Sofia, Faculty of Mathematics and
Informatics, blvd. James Bourchier 5, 1164, Sofia, Bulgaria, \&
Institute of Mathematics and Informatics, Bulgarian Academy of
Sciences}
\address{and Department of Mathematics,
University of Pennsylvania, DRL 209 South 33rd Street
Philadelphia, PA 19104-6395 } \email{ivanovsp@fmi.uni-sofia.bg}
\author{Ivan Minchev}
\address[Ivan Minchev]{University
of Sofia, Faculty of Mathematics and Informatics, blvd. James Bourchier 5, 1164 Sofia, Bulgaria;
Department of Mathematics and Statistics, Masaryk University, Kotlarska 2, 61137 Brno,
Czech Republic}
\email{minchev@fmi.uni-sofia.bg}
\author{Dimiter Vassilev}
\address[Dimiter Vassilev]{ Department of Mathematics and Statistics\\
University of New Mexico\\
Albuquerque, New Mexico, 87131-0001}
\email{vassilev@unm.edu}
\maketitle
\tableofcontents


\setcounter{tocdepth}{2}

\section{Introduction}
It is well known that the solution of the Yamabe problem on a
compact Riemannian manifold is unique in the case of negative or
vanishing scalar curvature. The proof of these results, which rely
on the maximum principle, extend readily to sub-Riemannian
settings such as the CR and quaternionic contact (abbr. qc) Yamabe
problems due to the sub-ellipticity of the involved operators. The
positive (scalar curvature) case is of continued interest since it
presents considerable difficulties due to the possible
non-uniqueness. The most important positive case in each of these
geometries is given by the corresponding round sphere due to its
role in the general existence theorem and also because of its
connection with the corresponding $L^2$ Sobolev type embedding
inequality. Through the corresponding Cayley transforms, the
sphere cases are equivalent to the problems of finding all
solutions to the respective Yamabe equation on the flat models
given by Euclidean space or Heisenberg groups. The Riemannian and
CR sphere cases were settled in \cite{Ob} and \cite{JL3}. It
should be noted  that the Euclidean case can be handled
alternatively by a reduction to a radially symmetric solution
\cite{GNN} and \cite{Ta}. Furthermore, \cite{Ob} established a
uniqueness result in every conformal class of an Einstein metric.
In this paper we solve the qc Yamabe problem on the $4n+3$
dimensional round sphere and quaternionic Heisenberg group and
establish a uniqueness result in every qc-conformal class
containing a 3-Sasakain metric.

We continue by giving a brief background and the statements of our results.
It is well known that the sphere at infinity of a any non-compact
symmetric space $M$ of rank one carries a natural
Carnot-Carath\'eodory structure, see \cite{M,P}. A quaternionic
contact (qc) structure, \cite{Biq1,Biq2}, appears naturally as the
conformal boundary at infinity of the quaternionic hyperbolic
space. {Following Biquard, a quaternionic contact structure
(\emph{qc structure}) on a real (4n+3)-dimensional manifold $M$ is
a codimension three distribution $H$ (\emph{the horizontal distribution}) locally given as the kernel of a $%
\mathbb{R}^3$-valued one-form $\eta=(\eta_1,\eta_2,\eta_3)$, such
that, the three two-forms $d\eta_i|_H$ are the fundamental forms
of a quaternionic Hermitian structure on $H$.
The 1-form $\eta$
is determined up to a conformal factor and the action of $SO(3)$ on $\mathbb{R}^3$, and therefore $%
H$ is equipped with a conformal class $[g]$ of quaternionic
Hermitian metrics. To every metric in the fixed conformal class
one can associate a linear connection with torsion preserving the
qc structure, see \cite{Biq1}, which is  called the Biquard
connection. } For a fixed metric in the conformal class of metrics
on the horizontal space one associates the horizontal Ricci-type
tensor of the Biquard connection, which is called the qc Ricci
tensor. This is a symmetric tensor \cite{Biq1} whose trace-free
part is determined by the torsion endomorphism of the Biquard
connection \cite{IMV} while the trace part is determined by the
scalar curvature of the qc-Ricci tensor, called the qc-scalar
curvature. It was shown in \cite{IMV} that the torsion
endomorphism of the Biquard connection is completely determined by
the trace-free part of the horizontal Ricci tensor whose vanishing
defines the class of qc-Einstein manifolds. A basic example of a
qc manifold
is  a 3-Sasakian space which can be defined as a $(4n+3)$%
-dimensional Riemannian manifold whose Riemannian cone is a
hyperK\"ahler manifold and the qc structure is induced from that
hyperK\"ahler structure. It was shown in \cite{IMV,IMV4} that the qc-Einstein manifolds
of positive qc-scalar curvature are exactly the locally 3-Sasakian manifolds, up to a multiplication with a constant factor and a $SO(3)$%
-matrix. In particular, every 3-Sasakian manifold has vanishing
torsion endomorphism and  is a qc-Einstein manifold.

The quaternionic contact Yamabe problem on a compact qc manifold
$M$ is the problem  of finding a metric  $\bar g\in [g]$
on $H$ for which the qc-scalar curvature is constant. A natural
question is to determine the possible uniqueness or non-uniqueness
of such qc-Yamabe metrics.

The question reduces to the solvability of the quaternionic
contact (qc) Yamabe equation \eqref{e:conf change
scalar curv}. Taking the conformal factor in the form $\bar\eta=u^{4/(Q-2)}%
\eta$, $Q=4n+6$, turns \eqref{e:conf change scalar curv} into the
equation
\begin{equation*}
\mathcal{L} u\ \equiv\ 4\frac {Q+2}{Q-2}\ \triangle u -\ u\, Scal
\ =\ - \ u^{2^*-1}\ \overline{Scal},
\end{equation*}
where $\triangle $ is the horizontal sub-Laplacian, $\triangle h\
=\ tr^g(\nabla^2h)$,  $Scal$ and $\overline{Scal}$ are the
qc-scalar curvatures
correspondingly of $(M,\, \eta)$ and $(M, \, \bar\eta)$, and $2^* = \frac {2Q%
}{Q-2},$ with $Q=4n+6$--the homogeneous dimension.

Another motivation for studying the qc Yamabe equation comes from
its connection with the determination  of the norm and extremals
in the $L^2$ Folland-Stein \cite{FS}
Sobolev-type embedding on the quaternionic Heisenberg group  $\boldsymbol{%
G\,(\mathbb{H})}$, \cite{GV}, \cite{Va1},  \cite{Va2} and
completed in \cite{IMV2}. The qc Yamabe equation is essentially
the Euler-Lagrange equation of the extremals for the $L^2$ case of the
Folland-Stein inequality \cite{FS} on the quaternionic Heisenberg group $\boldsymbol{%
G\,(\mathbb{H})}$.

{On a compact quaternionic contact manifold $M$ with a fixed
conformal class $[\eta]$ the qc Yamabe equation characterizes the
non-negative extremals of the qc Yamabe functional defined by
\begin{equation*}
\Upsilon (u)\ =\ \int_M\Bigl(4\frac {Q+2}{Q-2}\ \lvert \nabla u
\rvert^2\ +\ {Scal}\, u^2\Bigr) dv_g,\qquad \int_M u^{2^*}\,
dv_g \ =\ 1, \ u>0.
\end{equation*}
Here $dv_g$ denotes the Riemannian volume form of the
Riemannian metric
on $M$ extending in a natural way the horizontal metric associated to $\eta$. Considering $M$ equipped with a fixed qc structure, hence, a
conformal class $[\eta]$,} the Yamabe constant is defined as the
infimum
\begin{equation*}
\lambda(M)\ \equiv \ \lambda(M, [\eta])\ =\ \inf \{ \Upsilon (u)
:\ \int_M u^{2^*}\, dv_g \ =\ 1, \ u>0\}.
\end{equation*}
The main result of \cite{Wei} is that the qc Yamabe equation has a
solution on a compact qc manifold provided
$\lambda(M)<\lambda(S^{4n+3})$, where $S^{4n+3}$ is the standard
unit sphere in the  quaternionic space $\mathbb H^n$.

In this paper we consider the qc Yamabe problem on the {\  unit $(4n+3)$%
-dimensional sphere in $\mathbb H^n$. The standard 3-Sasaki
structure on the sphere $\tilde\eta$ has a constant qc-scalar
curvature $\widetilde{\text{Scal}}=16n(n+2)$ and vanishing
trace-free part of its qc-Ricci tensor, i.e., it is a qc-Einstein
space. The images under conformal quaternionic contact
automorphisms are again qc-Einstein structures  and, in
particular, have constant qc-scalar curvature. In \cite{IMV} we
conjectured that these are the only solutions to the Yamabe
problem on the quaternionic sphere and proved it in dimension
seven in \cite{IMV1}. One of the main goals of this paper is to
prove  this conjecture in full generality.
\begin{thrm}
\label{main2} Let $\tilde\eta=\frac{1}{2h}\eta$ be a qc conformal
transformation of the standard qc-structure $\tilde\eta$ on a
3-Sasakian sphere of dimension $4n+3$. If $\eta$ has constant
qc-scalar curvature, then up to a multiplicative constant $\eta$
is obtained from $\tilde\eta$ by a conformal quaternionic contact
automorphism.
\end{thrm}
We note that Theorem~\ref{main2} together with the results of
\cite{IMV} allows the determination of \emph{all} solutions of the
qc Yamabe problem on the sphere and on the quaternionic Heisenberg
group $\QH$. In fact, as a consequence of Theorem~\ref{main2}, we
obtain here that all solutions to the qc Yamabe equation are given
by the functions which realize the equality case of the $L^2$
Folland-Stein inequality found in \cite{IMV2} with the help
of the center of mass technique
developed for the CR case in \cite{FL} and \cite{BFM}.

Recall that the quaternionic Heisenberg group $\QH$ of homogeneous
dimension $Q=4n+6$ is given by
$\QH=\mathbb{H}^n\times\text{Im}\mathbb{H}$, $\quad
(q=(t^a,x^a,y^a,z^a)\in \mathbb H^n,\omega=(x,y,z)\in
\text{Im}\mathbb{H})$ with the group low
$$
(q_o, \omega_o)\circ(q, \omega)\ =\ (q_o\ +\ q, \omega\ +\
\omega_o\ + \ 2\ \text {Im}\
q_o\, \bar q). $$ 
The "\emph{standard}" qc contact form in quaternion variables is $
\tilde\Theta= (\tilde\Theta_1,\ \tilde\Theta_2, \ \tilde\Theta_3)=
\frac 12\ (d\omega - q \cdot d\bar q + dq\, \cdot\bar q) $.
The corresponding sub-Laplacian $
\triangle_{\tilde\Theta} u=\sum_{a=1}^n \left (T_{\alpha}^2u+
X_{\alpha}^2u+Y_{\alpha}^2u+Z_{\alpha}^2u \right ), $ where
$T_a,X_a,Y_a,Z_a$ denote the left-invariant horizontal vector
fields on $\QH$.
Theorem \ref{main2} shows, in particular, the following
\begin{cor}
If $\Phi$ satisfies the qc Yamabe equation on  the quaternionic
Heisenberg group $\QH$,
\begin{equation*}
\frac {4(Q+2)}{Q-2}\triangle_{\tilde\Theta} \Phi = -S_{\Theta}\,
\Phi^{2^*-1},
\end{equation*}
for some constant $S_{\Theta}$, then up to a left translation the
function $\Phi=(2h)^{-(Q-2)/4}$ and $h$ is given by
\begin{equation}\label{e:Liouville conf factor}
h(q,\omega) \ =\ c_0\ \Big [  \big ( \sigma\ +\
 |q+q_0|^2 \big )^2\  +\  |\omega\ +\ \omega_o\ +
\ 2\ \text {Im}\  q_o\, \bar q|^2 \Big ],
\end{equation}
for some fixed $(q_o,\omega_o)\in \QH$ and constants $c_0>0$ and
$\sigma>0$. Furthermore, the qc-scalar curvature of $\Theta$ is
$ 
S_{\Theta}=128n(n+2)c_0\sigma. $ 
\end{cor}

This confirms the Conjecture made after \cite[Theorem 1.1]{GV}. In
\cite[Theorem 1.6]{GV} the above result is proved on all groups of
Iwasawa type, but with the assumption of partial-symmetry of the
solution. Here with a completely different method from \cite{GV}
we show that the symmetry assumption is superfluous. The
corresponding solutions on the 3-Sasakain sphere are obtained via
the Cayley transform, see for example \cite{IMV,IMV1,IMV2},
\cite[Sections 2.3 \& 5.2.1]{IV3} for an account and history.
Finally, it should be observed that the functions
\eqref{e:Liouville conf factor} with $c_0\in \mathbb{R}$ give all
conformal factors for which $\Theta$ is also qc-Einstein. 

We derive Theorem \ref{main2} from a more general result in which we solve the qc Yamabe problem on a  locally 3-Sasakian
compact manifolds.
By the results of \cite{IMV} and \cite{IMV4} a qc-Einstein
manifold is of constant qc-scalar curvature, hence as far as the
qc Yamabe equation is concerned only the uniqueness of solutions
needs to be addressed. As mentioned earlier, the interesting case
is when the qc-scalar curvature is a positive constant, hence we
focus exclusively on the locally 3-Sasakian case.

\begin{thrm}\label{mainth}
Let $(M, \bar\eta)$ be a compact locally 3-Sasakian qc manifold of qc-scalar curvature $16n(n+2)$. If $\eta=2h\bar\eta$ is  qc-conformal to $\bar\eta$ structure which  is also of constant
qc-scalar curvature, then up to a homothety $(M,\eta)$ is locally 3-Sasakian
manifold. Furthermore, the function $h$ is constant unless
$(M,\bar\eta)$ is the unit 3-Sasakian sphere.
\end{thrm}

The proof of Theorem \ref{mainth} consists of two steps. The first step is a divergence formula
Theorem~\ref{t:div formulas} which shows that if $\bar\eta$ is of constant qc-curvature and is qc-conformal to a locally
3-Sasakian manifold, then $\bar\eta$ is also a locally 3-Sasakian manifold.
The general idea to search for such
a divergence formula goes back to Obata \cite{Ob} where the
corresponding result on a Riemannian manifold was proved for a
conformal transformation of an Einstein space.  However, our
result is motivated by the (sub-Riemannian)  CR case where a
formula of this type was introduced in the ground-breaking paper
of Jerison and Lee \cite{JL3}.
As far as the qc case is concerned in \cite[Theorem 1.2]{IMV} a
weaker results was shown, namely Theorem~\ref{mainth} holds
provided the vertical space of $\eta$ is integrable.  In dimension
seven, the $n=1$ case, this assumption was removed in
\cite[Theorem 1.2]{IMV1} where the result was established with the
help of a suitable divergence formula. The general case $n>1$
treated here  presents new difficulties due to the extra non-zero
torsion terms that appear in the higher dimensions,   which
complicate considerably the search of a suitable divergence
formula. In the seven dimensional case the [3]-component of the
traceless qc-Ricci tensor vanishes which decreases the number of
torsion components.

The proof of the second part of Theorem~\ref{mainth} builds on
ideas of Obata in the Riemannian case, who used that the gradient
of the (suitably taken) conformal factor is a conformal vector
field and the characterization of the unit sphere through its
first eigenvalue of the Laplacian among all Einstein manifolds. We
show a similar, although a more complicated relation between the
conformal factor and the existence of an infinitesimal qc
automorphism (qc vector field). Our divergence formula found in
Theorem~\ref{t:div formulas} involves a smooth function $f$, c.f.
\eqref{e:f}, expressed in terms of the conformal factor  and its
horizontal gradient. Remarkably, we found that the horizontal
gradient of $f$ is precisely the horizontal part of the  qc vector
field mentioned above and the sub-Laplacian of $f$ is an
eigenfuction of the sub-Laplacian with the smallest possible
eigenvalue $-4n$ thus showing a geometric nature of $f$ (cf
Remark~\ref{mys}). Then we use the characterization of the
3-Sasakian sphere by its first eigenvalue of the sub-Laplacian
among all locally 3-Sasakian manfolds established in
\cite[Theorem~1.2]{IPV1} for $(n>1)$ and in
 \cite[Corollary~1.2]{IPV2} for $n=1$.

\begin{rmrk}
Remarkably, a similar arguments also work in the CR case
describing the geometric nature of the mysterious function in the
Jerison-Lee's divergence formula in \cite{JL3}. Indeed, the
CR-Laplacian  of the real  part of the function $f$ defined in
\cite[Proposition~3.1]{JL3} turns out to be  an eigenfuction of
the CR-Laplacian with the smallest possible eigenvalue $-2n$ thus
showing a geometric nature of the real part of $f$.
\end{rmrk}

\begin{conv} \label{conven}
We use the following
\begin{itemize}

\item[1.] $\{e_1,\dots,e_{4n}\}$ denotes an orthonormal basis of
the horizontal space $H$; \item[2.] The capital letters X,Y,Z...
denote horizontal vectors, $X,Y,Z...\in H$.

\item[3.] The summation convention over repeated vectors from the basis $%
\{e_1,\dots,e_{4n}\}$ will be used. For example, for a
(0,4)-tensor $P$,  $k=P(e_b,e_a,e_a,e_b)$ means
$k=\sum_{a,b=1}^{4n}P(e_b,e_a,e_a,e_b).$

\item[4.] The triple $(i,j,k)$ denotes any cyclic permutation of
$(1,2,3)$.
\end{itemize}
\end{conv}

\textbf{Acknowledgements}  S.Ivanov is visiting University of
Pennsylvania, Philadelphia. S.I. thanks UPenn for providing the
support and an excellent research environment during the whole
stages of the paper. S.I. and I.M. are partially supported by
Contract DFNI I02/4/12.12.2014 and  Contract 168/2014 with the
Sofia University "St.Kl.Ohridski". I.M. is  supported by a SoMoPro II Fellowship which is co-funded  by
the European Commission\footnote{This article reflects only the author's views and the EU is not liable for any use that may be made of the information contained therein.} from \lq\lq{}People\rq\rq{} specific programme (Marie Curie Actions) within the EU
Seventh Framework Programme on the basis of the grant agreement REA No. 291782. It is further co-financed by the South-Moravian Region. D.V. was partially supported by Simons Foundation grant \#279381.


\section{Quaternionic contact manifolds}

\label{s:review} In this section we will briefly review the basic notions of
quaternionic contact geometry and recall some results from \cite{Biq1} and
\cite{IMV}, see \cite{IV3} for a more leisurely exposition.

A quaternionic contact (qc) manifold $(M, \eta,g, \mathbb{Q})$ is a $4n+3$%
-dimensional manifold $M$ with a codimension three distribution
$H$ locally given as the kernel of a 1-form
$\eta=(\eta_1,\eta_2,\eta_3)$ with values in $\mathbb{R}^3$. In
addition $H$ has an $Sp(n)Sp(1)$ structure, that is, it is
equipped with a Riemannian metric $g$ and a rank-three bundle
$\mathbb{Q}$ consisting of endomorphisms of $H$ locally generated
by three almost complex structures $I_1,I_2,I_3$ on $H$ satisfying
the identities of the imaginary unit quaternions,
$I_1I_2=-I_2I_1=I_3, \quad I_1I_2I_3=-id_{|_H}$ which are
hermitian compatible with the metric $g(I_s.,I_s.)=g(.,.)$ and the
following contact condition holds $$\qquad 2g(I_sX,Y)\ =\
d\eta_s(X,Y).$$


A special phenomena, noted in \cite{Biq1}, is that the contact
form $\eta$ determines the quaternionic structure and the metric
on the horizontal distribution in a unique way.

The transformations preserving a given quaternionic contact
structure $\eta$, i.e., $\bar\eta=\mu\Psi\eta$ for a positive
smooth function $\mu$ and an $SO(3)$ matrix $\Psi$ with smooth
functions as entries are called \emph{quaternionic contact
conformal (qc-conformal) transformations}.  If the function $\mu$
is constant $\bar\eta$ is called qc-homothetic to $\eta$. The qc
conformal curvature tensor $W^{qc}$, introduced in \cite{IV}, is
the obstruction for a qc structure to be locally qc conformal to
the standard 3-Sasakian structure on the $(4n+3)$-dimensional
sphere \cite{IMV,IV}.
\begin{dfn}
\label{d:3-ctct auto} A diffeomorphism $\phi$ of a QC manifold $(M,[g],%
\mathbb{Q})$ is called a \emph{conformal quaternionic contact
automorphism (conformal qc-automorphism)} if $\phi$ preserves the
QC structure, i.e.
\begin{equation*}
\phi^*\eta=\mu\Phi\cdot\eta,
\end{equation*}
for some positive smooth function $\mu$ and some matrix $\Phi\in
SO(3)$ with smooth functions as entries and
$\eta=(\eta_1,\eta_2,\eta_3)^t$ is a local 1-form considered as a
column vector of three one forms as entries.
\end{dfn}

On a qc manifold with a fixed metric $g$ on $H$ there exists a
canonical
connection defined first by O. Biquard in \cite{Biq1} when the dimension $(4n+3)>7$, and in \cite%
{D} for the 7-dimensional case. Biquard showed that there  is a
unique connection $\nabla$ with torsion $T$  and a unique
supplementary subspace $V$ to $H$ in $TM$, such that:
\begin{enumerate}[(i)]
\item $\nabla$ preserves the decomposition $H\oplus V$ and the $
Sp(n)Sp(1)$ structure on $H$, i.e. $\nabla g=0, \nabla\sigma
\in\Gamma( \mathbb{Q})$ for a section
$\sigma\in\Gamma(\mathbb{Q})$, and its torsion on $H$ is given by
$T(X,Y)=-[X,Y]_{|V}$; \item for $\xi\in V$, the endomorphism
$T(\xi,.)_{|H}$ of $H$ lies in $ (sp(n)\oplus sp(1))^{\bot}\subset
gl(4n)$; \item the connection on $V$ is induced by the natural
identification $ \varphi$ of $V$ with the subspace $sp(1)$ of the
endomorphisms of $H$, i.e. $ \nabla\varphi=0$.
\end{enumerate}
This canonical connection is also known as \emph{the Biquard
connection}. When the dimension of $M$ is
at least eleven \cite{Biq1} also described the supplementary distribution $V$%
, which is (locally) generated by the so called Reeb vector fields $%
\{\xi_1,\xi_2,\xi_3\}$ determined by
\begin{equation}  \label{bi1}
\eta_s(\xi_k)=\delta_{sk}, \qquad (\xi_s\lrcorner
d\eta_s)_{|H}=0,\qquad (\xi_s\lrcorner
d\eta_k)_{|H}=-(\xi_k\lrcorner d\eta_s)_{|H},
\end{equation}
where $\lrcorner$ denotes the interior multiplication.
 If the dimension of $%
M $ is seven Duchemin shows in \cite{D} that if we assume, in
addition, the existence of Reeb vector fields as in \eqref{bi1},
then   the Biquard result  holds. Henceforth, by a qc structure in
dimension $7$ we shall mean a qc structure satisfying \eqref{bi1}.

The fundamental 2-forms $\omega_s$ of the quaternionic contact
structure $Q$ are defined by
\begin{equation*}  
2\omega_{s|H}\ =\ \, d\eta_{s|H},\qquad
\xi\lrcorner\omega_s=0,\quad \xi\in V.
\end{equation*}

Notice that equations \eqref{bi1} are invariant under the natural
$SO(3)$
action. Using the triple of Reeb vector fields we extend the metric $g$ on $%
H $ to a metric $h$ on $TM$ by requiring
$span\{\xi_1,\xi_2,\xi_3\}=V\perp H \text{ and }
h(\xi_s,\xi_k)=\delta_{sk}. $  The Riemannian metric $h$ as well
as the Biquard connection do not depend on the action of $SO(3)$
on $V$, but both change  if $\eta$ is multiplied by a conformal
factor \cite{IMV}. Clearly, the Biquard connection preserves the
Riemannian metric on $TM, \nabla h=0$.


The properties of the Biquard connection are encoded in  the
torsion endomorphism $T_{\xi}\in(sp(n)+sp(1))^{\perp}$. We recall
the $Sp(n)Sp(1)$ invariant decomposition. An endomorphism $\Psi$
of $H$ can be decomposed with respect to the quaternionic
structure $(\mathbb{Q},g)$ uniquely into four $Sp(n)$-invariant
parts 
$\Psi=\Psi^{+++}+\Psi^{+--}+\Psi^{-+-}+\Psi^{--+},$ 
where the superscript $+++$
means commuting with all three $I_{i}$, $+--$ indicates commuting with $%
I_{1} $ and anti-commuting with the other two and etc. The two $Sp(n)Sp(1)$%
-invariant components $\Psi_{[3]}=\Psi^{+++}, \quad
\Psi_{[-1]}=\Psi^{+--}+\Psi^{-+-}+\Psi^{--+} $ are determined by
\begin{equation*}
\begin{aligned} \Psi=\Psi_{[3]} \quad \Longleftrightarrow 3\Psi+I_1\Psi
I_1+I_2\Psi I_2+I_3\Psi I_3=0,\\ \Psi=\Psi_{[-1]}\quad
\Longleftrightarrow \Psi-I_1\Psi I_1-I_2\Psi I_2-I_3\Psi I_3=0.
\end{aligned}
\end{equation*}
With a short calculation one sees that the $Sp(n)Sp(1)$-invariant
components are the projections on the eigenspaces of the Casimir
operator $\Upsilon =\ I_1\otimes I_1\ +\ I_2\otimes I_2\ +\
I_3\otimes I_3$ corresponding, respectively, to the eigenvalues
$3$ and $-1$, see \cite{CSal}. If $n=1$ then the space of
symmetric endomorphisms commuting with all $I_s$ is 1-dimensional,
i.e. the [3]-component of any symmetric endomorphism $\Psi$
on $H$ is proportional to the identity, $\Psi_{[3]}=-\frac{tr\Psi}{4}Id_{|H}$%
. Note here that each of the three 2-forms $\omega_s$ belongs to
its [-1]-component, $\omega_s=\omega_{s[-1]}$ and constitute a
basis of the Lie algebra $sp(1)$.

\subsection{The torsion tensor}
Decomposing the endomorphism $T_{\xi }\in (sp(n)+sp(1))^{\perp }$
into its symmetric part $T_{\xi }^{0}$ and skew-symmetric part
$b_{\xi },T_{\xi }=T_{\xi }^{0}+b_{\xi }$, O. Biquard shows in
\cite{Biq1} that the torsion $T_{\xi }$ is completely trace-free,
$tr\,T_{\xi }=tr\,T_{\xi }\circ I_{s}=0$, its symmetric part has
the properties $T_{\xi _{i}}^{0}I_{i}=-I_{i}T_{\xi _{i}}^{0}\quad
I_{2}(T_{\xi _{2}}^{0})^{+--}=I_{1}(T_{\xi _{1}}^{0})^{-+-},\quad
I_{3}(T_{\xi _{3}}^{0})^{-+-}=I_{2}(T_{\xi _{2}}^{0})^{--+},\quad
I_{1}(T_{\xi _{1}}^{0})^{--+}=I_{3}(T_{\xi _{3}}^{0})^{+--}$. The
skew-symmetric part can be represented as $b_{\xi _{i}}=I_{i}u$,
where $u$ is a traceless symmetric (1,1)-tensor on $H$ which
commutes with $I_{1},I_{2},I_{3}$. Therefore we have $T_{\xi
_{i}}=T_{\xi _{i}}^{0}+I_{i}u$. If $n=1$ then the tensor $u$
vanishes identically, $u=0$, and the torsion is a symmetric
tensor, $T_{\xi }=T_{\xi }^{0}$.

\noindent The two $Sp(n)Sp(1)$-invariant trace-free symmetric
2-tensors $T^0(X,Y)= g((T_{\xi_1}^{0}I_1+T_{\xi_2}^{0}I_2+T_{
\xi_3}^{0}I_3)X,Y)$, $U(X,Y) =g(uX,Y)$ on $H$, introduced in
\cite{IMV}, have the properties:
\begin{equation}  \label{propt}
\begin{aligned} T^0(X,Y)+T^0(I_1X,I_1Y)+T^0(I_2X,I_2Y)+T^0(I_3X,I_3Y)=0, \\
U(X,Y)=U(I_1X,I_1Y)=U(I_2X,I_2Y)=U(I_3X,I_3Y). \end{aligned}
\end{equation}
In dimension seven $(n=1)$, the tensor $U$ vanishes identically,
$U=0$.

\noindent These tensors determine completely the torsion endomorphism of the Biquard connection due to
the following identity \cite[Proposition~2.3]{IV} $%
4T^0(\xi_s,I_sX,Y)=T^0(X,Y)-T^0(I_sX,I_sY)$ which implies
\begin{equation*} 
4T(\xi_s,I_sX,Y)=4T^0(\xi_s,I_sX,Y)+4g(I_suI_sX,Y)=
T^0(X,Y)-T^0(I_sX,I_sY)-4U(X,Y).
\end{equation*}

\subsection{The qc-Einstein condition and Bianchi identities}

We explain briefly the consequences of the Bianchi identities and
the notion of qc-Einstein manifold introduced in \cite{IMV} since
it plays a crucial role in solving the Yamabe equation in the
quaternionic sphere (see \cite{IMV1} for dimension seven). For
more details see \cite{IMV}.

Let $R=[\nabla,\nabla]-\nabla_{[\ ,\ ]}$ be the curvature  of the
Biquard connection $\nabla$. The Ricci tensor and the scalar
curvature, called \emph{qc-Ricci tensor} and \emph{qc-scalar
curvature}, respectively, are defined by
\begin{equation*}  \label{e:horizontal ricci}
Ric(X,Y)={g(R(e_a,X)Y,e_a)},  \qquad
Scal=Ric(e_a,e_a)=g(R(e_b,e_a)e_a,e_b).
\end{equation*}
According to \cite{Biq1} the Ricci tensor restricted to $H$ is a
symmetric tensor. If the trace-free part of the qc-Ricci tensor is
zero we call the quaternionic structure \emph{a qc-Einstein
manifold} \cite{IMV}. It is shown in \cite{IMV} that the qc-Ricci
tensor is completely determined by the components of the torsion.
\noindent{\ Theorem~1.3, Theorem~3.12 and Corollary~3.14 in \cite{IMV} imply%
} 
that on a qc manifold $(M^{4n+3},g,\mathbb{Q})$ the qc-Ricci
tensor and the qc-scalar curvature satisfy
\begin{equation*}  \label{sixtyfour}
\begin{aligned} Ric(X,Y) \ & =\ (2n+2)T^0(X,Y)
+(4n+10)U(X,Y)+\frac{Scal}{4n}g(X,Y)\\ Scal\ & =\
-8n(n+2)g(T(\xi_1,\xi_2),\xi_3) \end{aligned}
\end{equation*}
Hence, the qc-Einstein condition is equivalent to the vanishing of
the torsion endomorphism of the Biquard connection and in this
case the qc scalar curvature is constant \cite{IMV,IMV4}. If $Scal
>0$ {\ the latter} holds exactly when the qc-structure is locally
3-Sasakian up to a multiplication by a constant and an
$SO(3)$-matrix with smooth entries.
{
We remind that a (4n+3)-dimensional
Riemannian manifold $(M,g)$ is called 3-Sasakian if the cone metric $%
g_N=t^2g+dt^2$ on $N=M\times \mathbb{R}^+$ is a hyperk\"ahler
metric, namely, it has holonomy contained in $Sp(n+1)$. The
3-Sasakian manifolds are Einstein with positive Riemannian scalar
curvature.}

The following vectors will be important for our considerations,
\begin{equation}  \label{d:A_s}
A_i\ =\ I_i[\xi_j, \xi_k],\qquad A\ = \ A_1\ +\ A_2\ +\ A_3.
\end{equation}
\noindent We  denote with the same letter the corresponding
horizontal 1-form and recall the action of $I_s$ on it,
\begin{equation*}
{A} (X)\ = \ g(I_1[\xi_2,\xi_3]+I_2[\xi_3,\xi_1]+I_3[\xi_1,\xi_2],
X),\quad I_sA(X)=-A(I_sX).
\end{equation*}
The horizontal divergence $\nabla^*P$ of a (0,2)-tensor field $P$
on $M$ with respect to Biquard connection is defined to be the
(0,1)-tensor field $
\nabla^*P(.)=(\nabla_{e_a}P)(e_a,.). $ 
We have from \cite[Theorem 4.8]{IMV} that on a
$(4n+3)$-dimensional QC manifold with constant qc-scalar curvature
the next identities hold
\begin{equation}  \label{div:To}
\nabla^*T^0=(n+2){A}, \qquad \nabla^*U=\frac{1-n}{2}{A}.
\end{equation}

For any smooth function $h$ on a qc manifold with constant qc
scalar curvature the following  formulas are valid \cite[Lemma~4.1
]{IMV1}
\begin{equation}\label{l:div of I_sA}
\begin{aligned} \nabla^*\, \Bigl (\sum_{s=1}^3 dh(\xi_s) I_sA_s\Bigr )\
 =\ \sum_{s=1}^3 \ \nabla dh\,(I_s e_a, \xi_s)A_s(e_a);\\ \nabla^*\, \Bigl (\sum_{s=1}^3 dh(\xi_s) I_sA \Bigr )\
 =\ \sum_{s=1}^3 \ \nabla dh\,(I_s e_a, \xi_s)A(e_a). \end{aligned}
\end{equation}


\subsection{Qc conformal transformations}

\label{s:conf transf}

Let $h$ be a positive smooth function on a qc manifold $(M, \eta)$. Let $%
\bar\eta=\frac{1}{2h}\eta$ be a conformal deformation of the qc structure $%
\eta$. We will denote the objects related to $\bar\eta$ by over-lining the
same object corresponding to $\eta$. Thus,  $d\bar\eta=-\frac{1}{2h^2}%
\,dh\wedge\eta\ +\ \frac{1}{2h\,}d\eta$ and $\bar
g=\frac{1}{2h}g$. The new triple
$\{\bar\xi_1,\bar\xi_2,\bar\xi_3\}$ is determined by the
conditions defining the Reeb vector fields as follows $\bar\xi_s\
=\ 2h\,\xi_s\ +\ I_s\nabla h$, where $\nabla h$ is the horizontal
gradient defined by $g(\nabla h,X)=dh(X)$.  The components of the
torsion tensor transform according to the following formulas from
\cite[Section 5]{IMV}
\begin{equation}  \label{e:U conf change}\begin{split}
\overline T^0(X,Y) \ =\ T^0(X,Y)\ +\ h^{-1}\,[\nabla dh]_{[sym][-1]}(X,Y), \\
\bar U(X,Y) \ =\ U(X,Y)\ +\ (2h)^{-1}[\ {\nabla dh}-2h^{-1}dh\otimes
dh]_{[3][0]}(X,Y),
\end{split}
\end{equation}
where the symmetric part is given by \begin{equation*}
\label{symdh} [\ {\nabla dh}]_{[sym]}(X,Y)\ =\ \ {\nabla dh}(X,Y)\
+ \ \sum_{s=1}^3 dh(\xi_s)\,\omega_s(X,Y)  \end{equation*} {and
$_{[3][0]}$ indicates the trace free part of the [3]-component of
the corresponding tensor. }

In addition, the qc-scalar curvature changes according to the
formula \cite{Biq1}
\begin{equation}  \label{e:conf change scalar curv}
\overline {\text{Scal}}\ =\ 2h\,(\text{Scal})\ -\ 8(n+2)^2\,h^{-1}|\nabla
h|^2\ +\ 8(n+2)\,\triangle h.
\end{equation}

\section{Qc conformal transformations on qc Einstein manifolds}
Throughout this section $h$ is a positive smooth function on a qc
manifold $(M, g, \mathbb{Q})$ {\ with constant qc-scalar curvature
$Scal=16n(n+2)$} and $\bar\eta\ =\ \frac{1}{2h}\, \eta$  is a qc
Einstein structure which is a conformal deformation of the qc
structure $\eta$. We recall some formulas from \cite{IMV1} which
we need here.

First we write the expressions of the 1-forms $A_s, A$ in terms of
$h$ (see \cite[Lemma~ ]{IMV1})
\begin{multline}  \label{e:A_s}
A_i(X)\ =\ -\frac12 h^{-2}dh(X)\ -\ \frac 12h^{-3}\lvert \nabla h
\rvert^2dh(X)  -\ \frac 12 h^{-1}\Bigl (\ {\nabla dh} (I_jX,
\xi_j)\  +\ \ {\nabla dh} (I_kX, \xi_k) \Bigr )\\ +\ \frac 12
h^{-2}\Bigl (dh(\xi_j)\,dh (I_jX)\ +\ dh(\xi_k)\,dh (I_kX) \Bigr )
+\ \frac 14 h^{-2}\Bigl ( \ {\nabla dh} (I_jX, I_j \nabla h)\ +\ \ {\nabla dh%
} (I_kX, I_k \nabla h) \Bigr ).
\end{multline}
Thus, we have also
\begin{multline}  \label{e:A}
A(X)\ =\ -\frac32 h^{-2}dh(X)\ -\ \frac 32h^{-3}\lvert \nabla h \rvert^2dh(X)
\\
-\ h^{-1}\sum_{s=1}^3\ {\nabla dh} (I_sX, \xi_s)\ +\
h^{-2}\sum_{s=1}^3dh(\xi_s)\,dh (I_sX)\ +\ \frac 12 h^{-2}\sum_{s=1}^3\ {%
\nabla dh} (I_sX, I_s \nabla h)\
\end{multline}




Second we consider the following one-forms
\begin{equation}  \label{d:D_s}
\begin{aligned} D_s(X)
=-\frac1{2h}\Big[T^0(X,\nabla h)+T^0(I_sX,I_s\nabla h)\Big]
\end{aligned}
\end{equation}
For simplicity, using the musical isomorphism, we will denote with
$D_1, \, D_2, \, D_3$ the corresponding (horizontal) vector
fields, for example \hspace{3mm} $g(D_1, X)=D_1(X)$. Using
\eqref{propt}, we set
\begin{equation}  \label{d:def of D}
D\ =\ D_1\ +\ D_2\ +\ D_3\ =\ - h^{-1}\,T^{0}(X,\nabla h).
\end{equation}
Setting $\bar T^0=0$ in \eqref{e:U conf change}, we obtain from
equations \eqref{d:D_s}  the expressions (cf. \cite{IMV1} or
\cite{IV})
\begin{equation}  \label{Ds}
\begin{aligned} D_i(X)\ =\ h^{-2}\, dh(\xi_i)\,dh(I_iX)+ \ \frac14h^{-2}\, \bigl [\ \nabla dh\, (X,\nabla h)+\ \nabla dh\, (I_iX,I_i\nabla h)
\\ - \ \nabla dh\, (I_jX,I_j\nabla h)\ -\ \nabla dh\,
(I_kX,I_k\nabla h)\bigr ] .
\end{aligned}
\end{equation}
The equalities \eqref{d:def of D} together with \eqref{Ds} yield
\cite[Lemma~4.2]{IMV1}
\begin{equation}\label{n21}
D(X)\ =\ \frac 14 h^{-2}\Bigl (3\ {\nabla dh} (X, \nabla h)\  -\
\sum_{s=1}^3\ {\nabla dh} (I_sX, I_s\nabla h) \Bigr )\  +\
h^{-2}\sum_{s=1}^3dh(\xi_s)\,dh (I_sX).
\end{equation}
Third, we consider the following one-forms (and corresponding
vectors)
\begin{equation*}  \label{d:F_s}
F_s(X)\ =\ - h^{-1}\, {T^0}(X,I_s\nabla h).
\end{equation*}
\noindent From the definition of $F_i$ and \eqref{d:D_s} we find
\begin{equation}\label{e:F_s by D_s}
F_i(X)\ =\ - h^{-1}{T^0}(X,I_i\nabla h)
=\ -D_i(I_iX)\ +\ D_j(I_i X)\ +\ D_k(I_iX).
\end{equation}
We recall the next divergence
formulas established in \cite[Lemma~4.2, Lemma~4.3]{IMV1} with the
help of the contracted second Bianchi identity \eqref{div:To}.
\begin{equation}\label{divD}
\nabla^*\, D\ =\ \lvert T^0 \rvert^2\ -h^{-1}g(dh,D)\ -\ h^{-1}
(n+2)\,g(dh,A).
\end{equation}
\begin{multline}\label{p:div of F_s}
\nabla^*\, \Bigl (\sum_{s=1}^3 dh(\xi_s) F_s\Bigr )\ =\ \sum_{s=1}^3 \Bigl [%
\ \nabla dh\, (I_se_a,\xi_s)F_s(I_se_a)\Bigr] \\
+ \ h^{-1}\sum_{s=1}^3 \Bigl[dh(\xi_s)dh (I_se_a)D(e_a)\ +(n+2)\,dh(\xi_s)dh
(I_s e_a)\, A(e_a)\Bigr ].
\end{multline}

\section{The divergence formula}
{\ Following is our main technical result.}{\ As mentioned in the
introduction, we were motivated to seek a divergence formula of
this type based on the Riemannian, CR and seven dimensional qc
cases of the considered problem. The main difficulty was to find a
suitable vector field with non-negative divergence containing the
norm of the torsion. The fulfilment of this task was facilitated
by the results of \cite{IMV}. In particular,
similarly to the CR case, but unlike the Riemannian case, we were
not able to achieve a proof based purely on the Bianchi
identities, see \cite[Theorem
4.8]{IMV}.} 

Using $\overline{Scal}=Scal=16n(n+2)$ in the Yamabe equation
\eqref{e:conf change scalar curv} we have
\begin{equation}  \label{n1}
\triangle h=2n-4nh +h^{-1}(n+2)|\nabla h|^2.
\end{equation}
The equation \eqref{e:U conf change}  in the case $\bar T^0=\bar
U=0$ and \eqref{n1} motivate the definition of    the following
symmetric (0,2) tensors


\begin{multline}  \label{n2}
\mathbf{D}(X,Y)=-T^0(X,Y)=\frac{h^{-1}}{4}\Big[3\nabla^2h(X,Y)-%
\sum_{s=1}^3\nabla^2h(I_sX,I_sY)
+4\sum_{s=1}^3dh(\xi_s)\omega_s(X,Y)\Big]
\end{multline}
\begin{multline}  \label{n3}
\mathbf{E}(X,Y)=-2U(X,Y)=\frac{h^{-1}}{4}\Big[\nabla^2h(X,Y)+%
\sum_{s=1}^3\nabla^2h(I_sX,I_sY)\Big] \\
-\frac{2h^{-2}}{4}\Big[dh(X)dh(Y)+\sum_{s=1}^3dh(I_sX)dh(I_sY)\Big] -\frac{h^{-1}}{4}%
\Big(2-4h+h^{-1}|\nabla h|^2\Big)g(X,Y).
\end{multline}
The one form $D$ defined in \eqref{d:def of D} and expressed
in terms of $h$ in  \eqref{n21} satisfies $D(X)=h^{-1}%
\mathbf{D}(X,\nabla h).$

Consider the 1-form $E(X)=h^{-1}\mathbf{E}(X,\nabla h)$. We obtain
from \eqref{n2} and \eqref{n3} the expression
\begin{equation}  \label{n31}
E(X)=\frac{h^{-2}}{4}\Big[\nabla^2h(X,\nabla h)+\sum_{s=1}^3\nabla^2h(I_sX,I_s\nabla h)+%
\Big(-2+4h-3h^{-1}|\nabla h|^2\Big)dh(X)\Big].
\end{equation}
We also define the (0,3)-tensors $\mathbb{D}$ and $\mathbb{E}$ by
\begin{multline}\label{ddd3}
\mathbb{D}(X,Y,Z)=-\frac{h^{-1}}{8}\Big[dh(X)T^0(Y,Z)+dh(Y)T^0(X,Z)\\+\sum_{s=1}^3dh(I_sX)T^0(I_sY,Z)+\sum_{s=1}^3dh(I_sY)T^0(I_sX,Z)\Big]
\end{multline}
\begin{multline}\label{eee3}
\mathbb{E}(X,Y,Z)=\frac{h^{-1}}{8}\Big\{dh(X)\mathbf{E}(Y,Z)+
dh(Y)\mathbf{E}(X,Z)\\+\sum_{s=1}^3dh(I_sX)\mathbf{E}(I_sY,Z)
+\sum_{s=1}^3dh(I_sY)\mathbf{E}(I_sX,Z)\Big\}.
\end{multline}
After this preparations we are ready to state the main result.
\begin{thrm}
\label{t:div formulas} {\ Suppose $(M^{4n+3},\eta)$ is a
quaternionic contact
structure conformal to a 3-Sasakian structure $(M^{4n+3},\bar\eta)$}, $%
\tilde\eta\ =\ \frac{1}{2h}\, \eta.$ If $Scal_{\eta}=Scal_{\tilde%
\eta}=16n(n+2)$, then with $f$ given by
\begin{equation}  \label{e:f}
f\ = \ \frac 12\ +\ h\ +\ \frac 14 h^{-1}\lvert \nabla h \rvert^2,
\end{equation}
\noindent the following identity holds
\begin{multline}  \label{divgen}
\nabla^*\Bigl(f(D+E)\ +\  \sum_{s=1}^3dh(\xi_s)I_sE \ +\ \sum_{s=1}^3 dh(\xi_s)\, F_s \ +\
4\sum_{s=1}^3 dh(\xi_s)I_sA_s \ -\ \frac {10}{3}\sum_{s=1}^3
dh(\xi_s)\,I_s A \Bigr)\\ =\ \Big(\frac12 +h\Big)\Big(\lvert T^0
\rvert^2+\lvert\textbf{E}\rvert^2\Big)+2h|\mathbb{D}+\mathbb{E}|^2
+\ h\,\langle Q V,\, V\rangle .
\end{multline}
where  $Q$ is equal to
\begin{equation*}
Q := \left[ {%
\begin{array}{ccccccc}
{\displaystyle \frac {5}{2} } & -{\displaystyle \frac {1}{2} } &-{\displaystyle \frac {1}{2} } &-{\displaystyle \frac {1}{2} } & -2 & -2 &-2 \\[2ex]
-{\displaystyle \frac {1}{2} } &{\displaystyle \frac {5}{2} }\, & -{\displaystyle \frac {1}{2} } & -{\displaystyle \frac {1}{2} } & {\displaystyle \frac {10}{3}} \, & - {\displaystyle \frac {2}{3%
}} \, & - {\displaystyle \frac {2}{3} } \, \\[2ex]
-{\displaystyle \frac {1}{2} } &-{\displaystyle \frac {1}{2} } &{\displaystyle \frac {5}{2} }\, & -{\displaystyle \frac {1}{2} } & - {\displaystyle \frac {2}{3}} \, & {\displaystyle \frac {10}{3%
}} \, & - {\displaystyle \frac {2}{3 }} \, \\[2ex]
-{\displaystyle \frac {1}{2} } & -{\displaystyle \frac {1}{2} } & -{\displaystyle \frac {1}{2} } & {\displaystyle \frac {5}{2} }\, & - {\displaystyle \frac {2}{3}} \, & - {\displaystyle \frac {2}{%
3}} \, & {\displaystyle \frac {10}{3}} \, \\[2ex]
-2 &{\displaystyle \frac {10}{3}} \, & - {\displaystyle \frac {2}{3 }} \, & - {%
\displaystyle \frac {2}{3}} \, & {\displaystyle \frac {22}{3}} \, & - {%
\displaystyle \frac {2}{3}} \, & - {\displaystyle \frac {2}{3}} \, \\[2ex]
-2 &- {\displaystyle \frac {2}{3}} \, & {\displaystyle \frac {10}{3 }} \, & - {%
\displaystyle \frac {2}{3}} \, & - {\displaystyle \frac {2}{3}} \, & {%
\displaystyle \frac {22}{3}} \, & - {\displaystyle \frac {2}{3}} \, \\[2ex]
-2 &- {\displaystyle \frac {2}{3}} \, & - {\displaystyle \frac {2 }{3}} \, & {%
\displaystyle \frac {10}{3}} \, & - {\displaystyle \frac {2}{3}} \, & - {%
\displaystyle \frac {2}{3} } \, & {\displaystyle \frac {22}{3}} \,%
\end{array}%
}  \right]
\end{equation*}
Here, $Q$ is a positive definite matrix with eigenvalues $1$,
$\frac92\pm\frac{\sqrt {73}}{2}$ and $\frac{11}{2}\pm\frac{\sqrt
{89}}{2}$ and

$V=(E, D_1, D_2, D_3,A_1,
A_2, A_3)$ with $E$, $D_s$, $A_s$ defined, correspondingly, in \eqref{n31} \eqref{d:D_s} and %
\eqref{d:A_s}.
\end{thrm}

\begin{proof} For the sake of making some formulas more compact, in the proof we will use sometimes the notation $XY=g(X,Y)$ for the product of two horizontal vector fields $X$ and $Y$ and the similar abbreviation for horizontal 1-forms.

We begin by recalling \eqref{n21}, \eqref{n31} and \eqref{e:A}, which imply
\begin{multline}  \label{ead1}
A(X)=\frac{3E(X)-D(X)}2-h^{-1}\sum_{s=1}^3\nabla^2h(I_sX,\xi_s) \\
+\frac32h^{-2}\sum_{s=1}^3dh(\xi_s)dh(I_sX)-\frac32h^{-2}\Big(\frac12+h+\frac14h^{-1}|%
\nabla h|^2\Big)dh(X).
\end{multline}
Using the function $f$ defined in \eqref{e:f}, we write \eqref{ead1} in the
form
\begin{equation}  \label{ead2}
2\sum_{s=1}^3\nabla^2h(I_sX,\xi_s)=h(3E(X)-D(X)-2A(X))+3h^{-1}\sum_{s=1}^3dh(%
\xi_s)dh(I_sX)-3h^{-1}fdh(X).
\end{equation}
The sum of \eqref{n21} and \eqref{n31} yields
\begin{equation}  \label{ed1}
(E+D)(X)=h^{-2}\nabla^2h(X,\nabla h)+h^{-2}\sum_{s=1}^3dh(\xi_s)dh(I_sX)+\frac{h^{-2}}4%
\Big(-2+4h-3h^{-1}|\nabla h|^2\Big)dh(X)\Big].
\end{equation}
Using \eqref{e:f} and \eqref{ed1}, we obtain
\begin{equation}  \label{df}
2\nabla_Xf
= h(E+D)(X) -h^{-1}\sum_{s=1}^3dh(\xi_s)dh(I_sX)+h^{-1}fdh(X).
\end{equation}
We calculate the divergences of $E$ using \eqref{div:To} as
follows
\begin{multline}  \label{divE}
\nabla^*E=2h^{-2}dh(e_a)U(e_a,\nabla h)- 2h^{-1}(\nabla_{e_a}U)(e_a,\nabla
h) - 2h^{-1}U(e_a,e_b)\nabla^2h(e_a,e_b)  \\
=-h^{-1}(1-n)A(\nabla h)+U(e_a,e_b)(-2h^{-1})\Big[%
\nabla^2h(e_a,e_b)-2dh(e_a)dh(e_b)\Big]+h^{-1}E(\nabla h) \\
=|\mathbf{E}|^2+h^{-1}dh(e_a)E(e_a)-h^{-1}(1-n)dh(e_a)A(e_a).
\end{multline}
Similarly, we have
\begin{multline}  \label{divIE}
-\nabla^*I_sE=2h^{-2}dh(e_a)U(I_se_a,\nabla h)+
2h^{-1}(\nabla_{e_a}U)(e_a,I_s\nabla h)
- 2h^{-1}U(I_se_a,e_b)\nabla^2h(e_a,e_b) \\
=h^{-1}(1-n)A(I_s\nabla h)+U(I_se_a,e_b)(-2h^{-1})\Big[%
\nabla^2h(e_a,e_b)-2dh(e_a)dh(e_b)\Big]+h^{-1}E(I_s\nabla h) \\
=U(I_se_a,e_b)U(e_a,e_b)-h^{-1}(1-n)dh(I_se_a)A(e_a)
=-h^{-1}(1-n)dh(I_se_a)A(e_a),
\end{multline}
since $U(I_se_a,e_b)U(e_a,e_b)=E(I_s\nabla h)=0$ due to
\eqref{propt}.

Now we are prepared to calculate the divergence of the first four terms. Using %
\eqref{divD}, \eqref{p:div of F_s}, \eqref{divE}, \eqref{df},
\eqref{divIE} and \eqref{ead2}, we have
\begin{multline}  \label{newdiv1}
\nabla_{e_a}\Big[f(D+E)(e_a)-\sum_{s=1}^3dh(\xi_s)E(I_se_a)+\sum_{s=1}^3dh(\xi_s)F_s(e_a)\Big] \\
= \Big(\frac{h}2(E+D)(e_a) -\frac{h^{-1}}2\sum_{s=1}^3dh(\xi_s)dh(I_se_a)+\frac{h^{-1}}%
2fdh(e_a)\Big)(D+E)(e_a) \\
+f\Big[-h^{-1}D(\nabla h)-h^{-1}(n+2)A(\nabla h)+|T^0|^2+
|E|^2+h^{-1}dh(e_a)E(e_a)-h^{-1}(1-n)dh(e_a)A(e_a)\Big] \\
+h^{-1}(1-n)\sum_{s=1}^3dh(\xi_s)dh(I_se_a)A(e_a)+\sum_{s=1}^3\nabla^2h(I_se_a,\xi_s)E(e_a)
\\
+ \sum_{s=1}^3\nabla^2h\, (I_se_a,\xi_s)F_s(I_se_a) + \
h^{-1}\sum_{s=1}^3dh(\xi_s)dh
(I_se_a)D(e_a)\ +(n+2)\,\sum_{s=1}^3dh(\xi_s)dh (I_s e_a)\, A(e_a) \\
= f(|T^0|^2+|\mathbf{E}|^2)+\frac{h}2|D+E|^2 +%
\frac{h}2(3E-D-2A)(e_a)E(e_a) \\
+h^{-1}\Big[\sum_{s=1}^3dh(\xi_s)dh(I_se_a)-fdh(e_a)\Big]\Big(\frac12D(e_a)+3A(e_a)\Big)
+ \sum_{s=1}^3\nabla^2 h\, (I_se_a,\xi_s)F_s(I_se_a).
\end{multline}

Applying  \eqref{l:div of I_sA} and \eqref{ead2} we obtain
\begin{multline}  \label{newdiv2}
\nabla_{e_a}\Big[f(D+E)(e_a)-\sum_{s=1}^3dh(\xi_s)E(I_se_a)+\sum_{s=1}^3dh(\xi_s)F_s(e_a)-2\sum_{s=1}^3dh(%
\xi_s)I_sA(e_a)\Big] \\
=
f(|T^0|^2+|\mathbf{E}|^2)+\frac{h}2|D+E|^2+\frac{h}2(3E-D-2A)E-h(3E-D-2A)A
\\
+\frac{h^{-1}}2\Big[\sum_{s=1}^3dh(\xi_s)dh(I_se_a)-fdh(e_a)\Big]D(e_a)
+ \sum_{s=1}^3\nabla^2 h\, (I_se_a,\xi_s)F_s(I_se_a)
\end{multline}

According to \eqref{e:F_s by D_s}, the last term in \eqref{newdiv2} reads
\begin{multline}  \label{newDs}
\sum_{s=1}^3\nabla^2 h\, (I_se_a,\xi_s)F_s(I_se_a)=D_1(e_a)\Big[\nabla^2h(I_1e_a,\xi_1)-%
\nabla^2h(I_2e_a,\xi_2)-\nabla^2h(I_3e_a,\xi_3) \Big] \\
+D_2(e_a)\Big[-\nabla^2h(I_1e_a,\xi_1)+\nabla^2h(I_2e_a,\xi_2)-%
\nabla^2h(I_3e_a,\xi_3) \Big] \\
+D_3(e_a)\Big[-\nabla^2h(I_1e_a,\xi_1)-\nabla^2h(I_2e_a,\xi_2)+%
\nabla^2h(I_3e_a,\xi_3) \Big].
\end{multline}
Using \eqref{newDs} we rewrite the last line in \eqref{newdiv2} as follows
\begin{multline}  \label{newDs1}
\Big[\frac{h^{-1}}2\sum_{s=1}^3dh(\xi_s)dh(I_se_a)-\frac{h^{-1}}2fdh(e_a)\Big]D(e_a)
+
\sum_{s=1}^3\nabla^2 h\, (I_se_a,\xi_s)F_s(I_se_a) \\
= D_1(e_a)\Big[\nabla^2h(I_1e_a,\xi_1)-\nabla^2h(I_2e_a,\xi_2)-%
\nabla^2h(I_3e_a,\xi_3)+\frac{h^{-1}}2\sum_{s=1}^3dh(\xi_s)dh(I_se_a)-\frac{h^{-1}}%
2fdh(e_a) \Big] \\
+D_2(e_a)\Big[-\nabla^2h(I_1e_a,\xi_1)+\nabla^2h(I_2e_a,\xi_2)-%
\nabla^2h(I_3e_a,\xi_3)+\frac{h^{-1}}2\sum_{s=1}^3dh(\xi_s)dh(I_se_a)-\frac{h^{-1}}%
2fdh(e_a) \Big] \\
+D_3(e_a)\Big[-\nabla^2h(I_1e_a,\xi_1)-\nabla^2h(I_2e_a,\xi_2)+%
\nabla^2h(I_3e_a,\xi_3)+\frac{h^{-1}}2\sum_{s=1}^3dh(\xi_s)dh(I_se_a)-\frac{h^{-1}}%
2fdh(e_a) \Big].
\end{multline}
The equalities  \eqref{n31}, \eqref{Ds} and \eqref{e:A_s} imply
\begin{multline}  \label{newAsdsE}
\nabla^2h(I_2X,\xi_2)+\nabla^2h(I_3X,\xi_3)\\=h(E-D_1-2A_1)(X)
+h^{-1}\sum_{s=1}^3dh(\xi_s)dh(I_sX)-h^{-1}fdh(X),
\end{multline}
Subtracting two times \eqref{newAsdsE} from \eqref{ead2} we obtain
\begin{multline}  \label{fin1}
\nabla^2h(I_1e_a,\xi_1)-\nabla^2h(I_2e_a,\xi_2)-\nabla^2h(I_3e_a,\xi_3)+%
\frac{h^{-1}}2\sum_{s=1}^3dh(\xi_s)dh(I_se_a)-\frac{h^{-1}}2fdh(e_a) \\
=\frac{h}2\Big[-E-D+4D_1-2A+8A_1\big](e_a)
\end{multline}
The left-hand side of the above identity  is the second line in
\eqref{newDs1}. The other two lines are evaluated similarly and
the formulas are obtained from the above by a cyclic rotation of
$\{1,2,3\}$.
 A substitution of the resulting new form of \eqref{newDs1}
 in \eqref{newdiv2} give
\begin{multline}  \label{finn}
\nabla_{e_a}\Big[f(D+E)(e_a)-\sum_{s=1}^3dh(\xi_s)E(I_se_a)+\sum_{s=1}^3dh(\xi_s)F_s(e_a)-2\sum_{s=1}^3dh(%
\xi_s)I_sA(e_a)\Big]
\\=
f\Big(|T^0|^2+|\mathbf{E}|^2\Big)+\frac{4h}2\Big[%
E^2+A^2+D_1^2+D_2^2+D_3^2-2AE+2A_1D_1+2A_2D_2+2A_3D_3\Big].
\end{multline}
In view of \eqref{l:div of I_sA} for any non-zero
constant $c$ we calculate the following divergences as follows
\begin{multline}  \label{divnew}
\nabla_{e_a}\Big(c\sum_{s=1}^3dh(\xi_s)I_sA_s(e_a) -\frac{c}3\sum_{s=1}^3dh(\xi_s)I_sA(e_a)\Big) \\
=\frac{c}3\Big[2\nabla^2h(I_1e_a,\xi_1)-\nabla^2h(I_2e_a,\xi_2)-%
\nabla^2h(I_3e_a,\xi_3) \Big]A_1(e_a) \\
+\frac{c}3\Big[2\nabla^2h(I_2e_a,\xi_2)-\nabla^2h(I_1e_a,\xi_1)-%
\nabla^2h(I_3e_a,\xi_3) \Big]A_2(e_a) \\
+\frac{c}3\Big[2\nabla^2h(I_3e_a,\xi_3)-\nabla^2h(I_2e_a,\xi_2)-%
\nabla^2h(I_1e_a,\xi_1) \Big]A_3(e_a)
\end{multline}
 subtracting
\eqref{newAsdsE} from twice  \eqref{ead2} yields
\begin{multline}  \label{divnew1}
2\nabla^2h(I_1e_a,\xi_1)-\nabla^2h(I_2e_a,\xi_2)-\nabla^2h(I_3e_a,\xi_3) 
\\=h\Big[2D_1-D_2-D_3+4A_1-2A_2-2A_3\Big](e_a)
\end{multline}
Now, taking into account \eqref{divnew1}, \eqref{divnew} and
\eqref{finn} we obtain
\begin{multline}  \label{finnn}
\nabla^*\Big[f(D+E)(X)-\sum_{s=1}^3dh(\xi_s)E(I_sX)+\sum_{s=1}^3dh(\xi_s)F_s(X)-2\sum_{s=1}^3dh(\xi_s)I_sA(X)\Big]\\+\nabla^*\Big[c\quad
\sum_{s=1}^3dh(\xi_s)I_sA_s(X) -\frac{c}3\sum_{s=1}^3dh(\xi_s)I_sA(X)\Big] \\
= f\Big(|T^0|^2+|\mathbf{E}|^2\Big)+\frac{4h}2\Big[%
E^2+A^2+D_1^2+D_2^2+D_3^2-2AE+2A_1D_1+2A_2D_2+2A_3D_3\Big] \\
+h\frac{c}3\Big[(2D_1-D_2-D_3+4A_1-2A_2-2A_3)A_1\Big]
+h\frac{c}3\Big[(2D_2-D_1-D_3+4A_2-2A_1-2A_3)A_2\Big] \\
+h\frac{c}3\Big[(2D_3-D_2-D_1+4A_3-2A_2-2A_1)A_3\Big]
\end{multline}

In the next Lemma, as in the proof of Theorem \ref{t:div formulas}, we shall use again the notation $XY=g(X,Y)$ for the product of two horizontal vector fields $X$ and $Y$ and the similar abbreviation for horizontal 1-forms.

\begin{lemma}\label{lemDE}
For the (0,3)-tensors $\mathbb{D}$ and $\mathbb{E}$ defined by
\eqref{ddd3} and \eqref{eee3} we have

\begin{equation}\label{calDE}
\begin{aligned}
|\mathbb{D}|^2\ =\ \frac{1}{8}h^{-2}|\nabla h|^2|T^0|^2-\frac14\sum_{s=1}^3|D_s|^2+\frac12(D_1D_2+D_1D_3+D_2D_3),\\
|\mathbb{E}|^2\ =\ \frac{1}{8}h^{-2}|\nabla
h|^2|\mathbf{E}|^2-\frac14|E|^2,\qquad \mathbb{D}\mathbb{E}\ =\
\frac14\sum_{s=1}^3 ED_s.\qquad
\end{aligned}
\end{equation}
Consequently,
\begin{multline}\label{ednew}
\frac14h^{-2}|\nabla h|^2(|T^0|^2+|\mathbf{E}|^2)\ =\
2|\mathbb{D}+\mathbb{E}|^2-\sum_{s=1}^3 ED_s\\\ +\
\frac12|E|^2+\frac12\sum_{s=1}^3|D_s|^2-(D_1D_2+D_1D_3+D_2D_3)
\end{multline}
\end{lemma}

\begin{proof}  We shall repeatedly apply \eqref{propt}, the defining equations \eqref{ddd3}, \eqref{eee3}, \eqref{d:A_s} and \eqref{d:def of D}. We have
\begin{multline}\label{d2}
|\mathbb{D}|^2\ 
=\frac{h^{-2}}{8}|\nabla
h|^2|T^0|^2+\frac{h^{-2}}{8^2}\Big(2T^0(\nabla h,e_c)T^0(\nabla
h,e_c)\\-4\sum_{s=1}^3T^0(I_s\nabla h,e_c)T^0(I_s\nabla h,e_c)
+2\sum_{s,t=1}^3T^0(I_sI_t\nabla h,e_c)T^0(I_tI_s\nabla
h,e_c)\Big)\\
=\frac{h^{-2}}{8}|\nabla
h|^2|T^0|^2+\frac{1}{4}\Big(-\sum_{s=1}^3D_s^2+2(D_1D_2+D_1D_3+D_2D_3)\Big)
\end{multline}



which is the first line of \eqref{calDE}. For example, the third
term in \eqref{d2} is calculated as follows
\begin{multline*}
\sum_{s,t=1}^3T^0(I_sI_t\nabla h,e_c)T^0(I_tI_s\nabla h,e_c)=
\sum_{s=1}^3\Big[T^0(\nabla h,e_c)T^0(\nabla h,e_c)-2T^0(I_s\nabla h,e_c)T^0(I_s\nabla h,e_c)\Big]\\
=6|D|^2-12\sum_{s=1}^3D_s^2+8(D_1D_2+D_1D_3+D_2D_3)
=-6\sum_{s=1}^3D_s^2+20(D_1D_2+D_1D_3+D_2D_3).
\end{multline*}


Similarly, we obtain the second line of \eqref{calDE}. The
equality \eqref{ednew} follows from \eqref{calDE} which completes
the proof of Lemma~\ref{lemDE}.
\end{proof}

Finally, the proof of Theorem~\ref{t:div formulas} follows by
letting $c=4$ in \eqref{finnn}and using \eqref{ednew} and
\eqref{d:A_s}.
\end{proof}

\section{Proof of Theorem\protect[\ref{mainth}] and Theorem\protect[\ref{main2}]}

We begin with the proof of Theorem \ref{mainth}. The first step of
the proof relies on Theorem~\ref{t:div formulas}. By a homothety
we can suppose that both qc-scalar curvatures are equal to
$16n(n+2)$. Integrating the divergence formula of
Theorem~\ref{t:div formulas}
and then using  the divergence theorem established in \cite[Proposition 8.1%
]{IMV} shows that the integral of the left-hand side is zero. Thus, the
right-hand side vanishes as well, which shows that the
quaternionic contact structure $\bar\eta$ has vanishing torsion,
i.e., it is also qc-Einstein according to
\cite[Proposition~4.2]{IMV}. This proves the first part of
Theorem~\ref{mainth}.

To prove the second part, we develop a sub-Riemannian extension of
the result of \cite{Ob}, see also \cite{BoEz87} and the review
\cite[Theorem 2.6]{IV14}, on the relation between the Yamabe
equation and the Lichnerowicz-Obata first eigenvalue estimate.  We
begin by recalling some results from \cite[Section 7.2]{IMV}. A
vector field $Q$ on a qc manifold $(M, \eta)$ is a \emph{qc vector
field}\index{qc vector field} if its flow preserves the horizontal
distribution $H=\ker \eta$. Since the conformal class of the qc
structure on $\text{span}\{\eta_1,\eta_2,\eta_3\}$ is uniquely
determined by $H$ (cf. \cite{Biq1})
, we have that
\begin{equation*}
\LieQ\eta=(\nu I+O)\cdot\eta,
\end{equation*}
where $\nu$ is a smooth function and $O\in so(3)$ is a matrix
valued function with smooth entries. Since the exterior derivative
$d$ commutes with the Lie derivative $\LieQ$, any qc vector field
$Q$ satisfies
\begin{gather*}
\LieQ g =\nu g,\qquad \LieQ I=O\cdot I,
\qquad I=(I_1,I_2,I_3)^t,
\end{gather*}
which is equivalent to saying that the flow of $Q$ preserves  the
conformal class $[g]$ of the horizontal metric and the
quaternionic structure $\mathbb Q$ on $H$. The function $\nu$ can
be easily expressed in terms of the divergence (with respect to
$g$) of the horizontal part $Q_H$ of the vector field $Q$. Indeed,
from \cite[Lemma 7.12]{IMV} we have
\begin{gather*}
g(\nabla_XQ_H,Y)\ +\ g(\nabla_YQ_H,X)\ + \ 2\eta_s(Q)g(T^0_{\xi_s} X,Y)=\nu\, g(X,Y),
\end{gather*}
hence
\begin{equation*}
\nu=\frac {1}{2n}\nabla^* Q_H.
\end{equation*}
This gives a geometric interpretation for the quantity $(\nabla^*
Q_H)$, namely, the flow of a qc vector field $Q$ preserves a fixed
metric $g\in[g]$ if and only if $\nabla^* Q_H=0$.

As an infinitesimal version of the qc Yamabe equation we obtain the
following general fact concerning the divergence of a QC vector
field.

\begin{lemma}\label{l:lapdiv} Let $(M,\eta)$ be a  qc  manifold.
For any qc vector field $Q$ on $M$ we have
\begin{equation*}
\Delta(\nabla^*Q_H)\ =\ -\ \frac{n}{2(n+2)}Q(\text{Scal})\ -\ \frac{\text{Scal}}{4(n+2)}\nabla^*Q_H,
\end{equation*}
where Scal, $\nabla^*$, $\Delta$ and the projection $Q_H$ correspond to the contact form $\eta$.
\end{lemma}

\begin{proof}
Suppose $Q$ is a qc vector field and
let $\phi_t$ be the corresponding (local) 1-parameter group of diffeomorphisms  generated by its flow. Then $$\phi_t^*(\eta)\ =\ {\displaystyle \frac{1}{ 2h_t}}\,\eta\quad\text{and}\quad \phi_t^*(g)\ =\ {\displaystyle \frac{1}{ 2h_t}}\,g$$ for some positive function $h_t$, depending smoothly on the parameter $t$. The qc scalar curvature $\text{Scal}_t$ of the pull back contact form $\phi_t^*(\eta)$ is given by $\text{Scal}_t=\text{Scal}\circ\phi_t$. Then, formula  \eqref{e:conf change scalar curv} yields
\begin{equation}\label{e:conf change scalar curv-t}
\text{Scal}\circ\phi_t\ =\ 2h_t\,(\text{Scal})\ -\ 8(n+2)^2\,h_t^{-1}|\nabla
h_t|^2\ +\ 8(n+2)\,\triangle h_t.
\end{equation}
Clearly, we have $h_0=\frac12$, and from
\begin{gather*}
\frac{1}{2n}(\nabla^*Q_H)\,g\ =\ \mathcal L_Q\,g\ =\ \frac{d}{dt}\vert_{ t=0}\left(\frac{1}{2h_t}g\right)\ =\ -\ \frac{h'_0}{2h_0}\,g\ =\ -\ 2h'_0\, g
\end{gather*}
we obtain that $h'_0=-\frac{1}{4n}\nabla^*Q_H$, where $h'_0$ denotes the derivative of $h_t$ at $t=0$.
A differentiation at $t=0$ in \eqref{e:conf change scalar curv-t} gives the lemma.
\end{proof}

\begin{lemma}\label{l:qc conf Einst} Let $(M,\eta)$ and $(M,\bar\eta)$ be qc-Einsten manifolds with equal qc-scalar curvatures
$16n(n+2)$. If $\eta$ and $\bar\eta$ are qc conformal to each other, $\overline \eta=\frac{1}{2h}\eta$ for some smooth positive function $h$, then  \begin{equation}\label{e:qc field}
Q=\frac12\nabla f+\sum_{s=1}^3dh(\xi_s)\xi_s
\end{equation}
is a qc vector field on $M$, where the function $f$ is defined in
\eqref{df}.
\end{lemma}

\begin{proof} The assumption of the lemma implies that
$E=D=D_s=A_s=0$. Using \eqref{newAsdsE},
\eqref{fin1} and \eqref{df} we obtain
$\nabla^2h(I_sX,\xi_s)=-df(X)$ and thus $\nabla^2h(X,\xi_s)=df(I_sX).$ It follows that
\begin{equation*}
\sum_{s=1}^3\nabla_X(dh(\xi_s)\xi_s)=\sum_{s=1}^3df(I_sX)\xi_s.
\end{equation*}

To show that the flow of the vector field $Q$, defined by
\eqref{e:qc field}, preserves the horizontal distribution $H$, for
any $X\in H$, we have

\begin{multline*}
\mathcal L_Q(X)
= \frac12 \,[\nabla f,X]+\sum_{s=1}^3[dh(\xi_s)\xi_s,X]  =
\frac12\nabla_{\nabla f} X\ -\ \frac12\nabla_X(\nabla f)\ -\
\sum_{s=1}^3\omega_s(\nabla f,X)\xi_s\ \\+\ \sum_{s=1}^3 \left
[dh(\xi_s)\nabla_{\xi_s}X- \nabla_X(dh(\xi_s)\xi_s)-
dh(\xi_s)T_{\xi_s}(X)\right ]  =
 \frac12\nabla_{\nabla f} X\ -\
\frac12\nabla_X(\nabla f)+\sum_{s=1}^3 dh(\xi_s)\nabla_{\xi_s}X \ \in H.
\end{multline*}

\end{proof}

At this point we are ready to complete the proof of Theorem
\ref{mainth}. Consider the qc vector field $Q$ defined in Lemma
\ref{l:qc conf Einst}. By Lemma \ref{l:lapdiv}, {the function
$\phi=\frac12\triangle f$ is either an eigenfunction of the
sub-Laplacian with eigenvalue $-4n$, $\triangle\phi=-4n\phi$}, or
it vanishes identically. In the first case, using the quaternionic
contact version of the Lichnerowicz-Obata eigenfunction sphere
theorem \cite[Theorem
1.2]{IPV2} and \cite[Corollary~1.2]{IPV1} (see also 
\cite{BauKim14}), we conclude
that $(M, \eta)$ is the 3-Sasakain sphere.
{In the other case, we have that $\Delta f =  0$, hence the
function $f=\frac 12 + h+\frac 14 h^{-1}\lvert \nabla h
\rvert^2=const$ since $M$ is compact. It follows that $h=1/2$ by
considering the points where $h$ achieves its minimum and maximum
and taking into account the qc Yamabe equation \eqref{n1}. }The
proof of Theorem~\ref{mainth} is complete.

\begin{rmrk}\label{mys} {Lemma~\ref{l:qc conf Einst} provides also a certain geometric
insight for the mysterious function $f$ in \eqref{e:f}. In fact,
up to an additive constant, $f$ is the unique function on $M$ for
which $Q_H=\frac 12\nabla f$ is the horizontal part of a qc vector
field $Q$ with vertical part $Q_V=dh(\xi_s)\xi_s$, $Q=Q_H+Q_V$.
This assertion is an easy consequence of the computation  given in
the proof of  Lemma~\ref{l:qc conf Einst}. Moreover, it implies
that on the 3-Sasakain sphere $\phi=\triangle f$ is an
eigenfuction of the sub-Laplacian realizing the smallest possible
eigenvalue $-4n$ on a compact locally 3-Sasakian manifold.}
\end{rmrk}

Theorem~\ref{main2} is a direct corollary from Theorem
\ref{mainth}. Alternatively, as in the proof of
Theorem~\ref{mainth}, we can use in the first step  Theorem
\ref{t:div formulas} which shows that the "new" structure is also
qc-Einstein. The second step of the proof of Theorem \ref{main2}
follows then also by taking into account \cite[Theorem 1.2]{IMV}
where all locally 3-Sasakian structures of positive constant
qc-scalar curvature  which are qc-conformal to the standard
3-Sasakian structure on the sphere were classified (we note that
this classification  extends easily to the case when no sign
condition of the "new" qc-structure is assumed, see \cite{IV14}).


\begin{thebibliography}{Biq1}

\bibitem{BauKim14}
 Baudoin, F., \& Kim, B., \emph{Sobolev, Poincar\'e, and isoperimetric inequalities for subelliptic diffusion operators satisfying a generalized curvature dimension inequality}. Rev. Mat. Iberoam. 30 (2014), no. 1, 109--131.

\bibitem{Biq1} Biquard, O., \emph{M\'{e}triques d'Einstein
asymptotiquement sym\'{e}triques}, Ast\'{e}risque \textbf{265} (2000).

\bibitem{Biq2} Biquard, O., \emph{Quaternionic contact structures},
Quaternionic structures in mathematics and physics (Rome, 1999), 23--30
(electronic), Univ. Studi Roma "La Sapienza", Roma, 1999.

\bibitem{BoEz87}
 Bourguignon, J.-P., \$ Ezin, J.-P., \emph{Scalar curvature functions in a conformal class of metrics and conformal transformations}.
Trans. Amer. Math. Soc. 301 (1987), no. 2, 723--736.

\bibitem{BFM}
Branson, T. P., Fontana, L., \& Morpurgo, C.,
\emph{Moser-Trudinger and Beckner-Onofri's inequalities on the CR
sphere},  Ann. of Math. (2) 177 (2013), no. 1, 1--52.


\bibitem{CSal} Capria, M. \& Salamon, S., \emph{Yang-Mills fields on
quaternionic spaces} Nonlinearity \textbf{1} (1988), no. 4, 517--530.






\bibitem{D} Duchemin, D., \emph{Quaternionic contact structures in
dimension 7}, Ann. Inst. Fourier (Grenoble) \textbf{56} (2006), no. 4,
851--885.





\bibitem{FS}  Folland, G. B. \& Stein, E. M., \emph{Estimates for the $%
\Bar {\partial}_{b}$ Complex and Analysis on the Heisenberg Group}, Comm.
Pure Appl. Math., \textbf{27}~(1974), 429--522.


\bibitem{FL}  Frank, R. L. \&  Lieb, E. H., \emph{Sharp constants in several inequalities on the Heisenberg group},
 Ann. of Math. (2) \textbf{176} (2012), no. 1, 349�381.


\bibitem{GV} Garofalo, N. \& Vassilev, D.,\ \emph{\ Symmetry properties
of positive entire solutions of Yamabe type equations on groups of
Heisenberg type}, Duke Math J, \textbf{106} (2001), no. 3, 411--449.


\bibitem{GNN}
Gidas, B., \ Ni, W. M.\   \&\  Nirenberg, L., \emph{Symmetry and related properties via the maximum principle},
Comm. Math. Phys., \textbf{68}~(1979), 209--243.


\bibitem{IMV} Ivanov, St., Minchev, I., \& Vassilev, D.,\ \emph{%
Quaternionic contact Einstein structures and the quaternionic
contact Yamabe problem}, Memoirs  Amer. Math. Soc. \textbf{vol.
231}, number 1086; http://dx.doi.org/10.1090/memo/1086


\bibitem{IMV1} Ivanov, St., Minchev, I., \& Vassilev, D., \emph{Extremals for the Sobolev inequality on the
seven dimensional quaternionic Heisenberg group and the
quaternionic contact Yamabe problem}, J. Eur. Math. Soc. (JEMS) 12
(2010), no. 4, 1041--1067.

\bibitem{IMV2} Ivanov, St., Minchev, I., \& Vassilev, D.,\ \emph{%
The optimal constant in the $L^2$ Folland-Stein inequality on the
quaternionic Heisenberg group}, Ann. Sc. Norm. Super. Pisa Cl.
Sci. (5) \textbf{Vol. XI} (2012), 635-652;

\bibitem{IMV4} Ivanov, S., Minchev, I., \& Vassilev, \emph{Quaternionic contact
Einstein structures}, arXiv:1306.0474, to appear in Math. Res. Lett.



\bibitem{IPV2} Ivanov, S., Petkov, A., \& Vassilev, D.,
\emph{The sharp lower bound of the first eigenvalue of the sub-Laplacian
on a quaternionic contact manifold in dimension 7}, Nonlinear Anal-Theory 93 (2013) 51--61.

\bibitem{IPV1}\bysame, \emph{The sharp
lower bound of the first eigenvalue of the sub-Laplacian on a
quaternionic contact manifold}, J. Geom. Anal. 24 (2014), no. 2, 756--778.




\bibitem{IV} Ivanov, St., \& Vassilev, D., \emph{Conformal quaternionic
contact curvature and the local sphere theorem}, J. Math. Pures
Appl. \textbf{93} (2010), 277--307.

\bibitem{IV3} \bysame, \emph{Extremals for the Sobolev Inequality and the
Quaternionic Contact Yamabe Problem}, World Scientific Publishing Co. Pte. Ltd., Hackensack, NJ, 2011. xviii+219 pp.

\bibitem{IV14}
\bysame, \emph{The Lichnerowicz and Obata  first eigenvalue
theorems and the Obata uniqueness result in the Yamabe problem on
CR and quaternionic contact manifolds}, preprint 2014, arxiv15


\bibitem{JL3} Jerison, D., \& Lee, J., \emph{Extremals for the Sobolev inequality on
the Heisenberg group and the CR Yamabe problem}, J. Amer. Math.
Soc., \textbf{1}~(1988), no. 1, 1--13.


\bibitem{M} Mostow, G. D., \emph{Strong rigidity of locally symmetric
spaces}, Annals of Mathematics Studies, No. 78. Princeton University Press,
Princeton, N.J.; University of Tokyo Press, Tokyo, 1973. v+195 pp.

\bibitem{Ob} Obata, M., \emph{The conjecture of conformal
transformations in Riemannian manifolds}, J. Diff. Geom., \textbf{6} (1971),
247--258.

\bibitem{P} Pansu, P., \emph{M\'etriques de Carnot-Carath\'eodory et
quasiisom\'etries des espaces sym\'etriques de rang un}, Ann. of Math. (2)
129 (1989), no. 1, 1--60.

\bibitem{Ta}
 Talenti,  G., \emph{Best constant in the Sobolev inequality}, Ann. Mat. Pura Appl., \textbf{110}~(1976), 353--372.

\bibitem{Va2} Vassilev, D, \textit{Yamabe type equations on Carnot groups%
}, Ph. D. thesis Purdue University, 2000.

\bibitem{Va1} Vassilev, D., \textit{Regularity near the characteristic
boundary for sub-laplacian operators}, Pacific J Math, \textbf{227} (2006),
no.~2, 361--397.

\bibitem{Wei} Wang, W., \emph{The Yamabe problem on quaternionic contact
manifolds},  Ann. Mat. Pura Appl., \textbf{186} (2007), no. 2, 359--380.

\end{thebibliography}
\end{document}